\documentclass[12pt, reqno]{amsart}
\usepackage{ amsmath,amsthm, amscd, amsfonts, amssymb, graphicx, color}
\usepackage[bookmarksnumbered, colorlinks, plainpages]{hyperref}
\newtheorem{thm}{Theorem}[section]
\newtheorem{cor}[thm]{Corollary}
\newtheorem{lem}[thm]{Lemma}

\newtheorem{exam}[thm]{Example}
\numberwithin{equation}{section}

\begin{document}

\title{Cline's formula for g-Drazin inverses}

\author{Huanyin Chen}
\author{Marjan Sheibani$^*$}
\address{
Department of Mathematics\\ Hangzhou Normal University\\ Hang -zhou, China}
\email{<huanyinchen@aliyun.com>}
\address{Women's University of Semnan (Farzanegan), Semnan, Iran}
\email{<sheibani@fgusem.ac.ir>}

\thanks{$^*$Corresponding author}


\begin{abstract}
Cline's formula for the well known generalized inverses such as Drazin inverse, generalized Drazin inverse is extended to the case when
$a(ba)^2=abaca=acaba=(ac)^2a$. Applications are given to some interesting Banach space operators.\end{abstract}

\maketitle

\section{Introduction}

Let $R$ be an associative ring with an identity. The commutant of $a\in R$ is defined by $comm(a)=\{x\in
R~|~xa=ax\}$. The double commutant of $a\in R$ is defined by $comm^2(a)=\{x\in R~|~xy=yx~\mbox{for all}~y\in comm(a)\}$.

An element $a\in R$ has Drazin inverse in case there exists $b\in R$ such that $$b=bab, b\in comm^2(a), a-a^2b\in R^{nil}.$$ The preceding $b$ is unique if exists, we denote it by $a^D$. Let $a,b\in R$. Then $ab$ has Drazin inverse if and only if $ba$ has Drazin inverse and $(ba)^{D}=b((ab)^{D})^2a$. This was known as Cline's formula for Drazin inverses (see \cite{CC}).

An element $a\in R$ has g-Drazin inverse (i.e., generalized Drazin inverse) in case there exists $b\in R$ such that $$b=bab, b\in comm^2(a), a-a^2b\in R^{qnil}.$$ The preceding $b$ is unique if exists, we denote it by $a^d$. Let $a,b\in R$. Then $ab$ has g-Drazin inverse if and only if $ba$ has g-Drazin inverse and $(ba)^{d}=b((ab)^{d})^2a$. This was known as Cline's formula for g-Drazin inverses (see \cite{C}).

Following Wang and Chen, an element $a$ in a ring $R$ has p-Drazin inverse if there exists $b\in comm^2(a)$ such that $b=b^2a, (a-a^2b)^k\in J(R)$. The p-Drazin inverse $b$ is also unique, and we denote it by $a^{pD}$ (see \cite{WC} ).

We shall extend Cline's formula for Drazin inverse, generalized Drazin inverse of $ba$ in a
ring when $ac$ has a corresponding inverse, $a(ba)^2=abaca=acaba=(ac)^2a$. This also recovers some recent results (see \cite{L}).

In Section 2, we  extend the Cline's formula for generalized inverse. We prove that for a ring $R$, if  $a(ba)^2=abaca=acaba=(ac)^2a$, for some $a, b, c\in R$ then,  $ac\in R^{d}$ if and only if $ba\in R^{d}$.

In Section 3, we generalized the Jacobson's Lemma and prove that if If $a(ba)^2=abaca=acaba=(ac)^2a$ in a ring $R$, then $$1-ac\in U(R)\Longleftrightarrow 1-ba\in U(R).$$ Also we study the common spectral properties of bounded linear operators.

Throughout the paper, all rings are associative with an identity. We use $R^{nil}, R^{qnil}$ and $R^{rad}$ to denote the set of all nilpotents, quasinilpotents and Jacobson radical of the ring $R$, respectively. $U(R)$ is the set of all units in $R$. $R^{D}$ and $R^{d}$ denote the sets of all elements in $R$ which have Drazin and g-Drazin inverses. ${\Bbb N}$ stands for the set of all natural numbers.

\section{Cline's Formula}

In \cite [Lemma 2.2]{L} proved that $ab\in R^{qnil}$ if and only if $ba\in R^{qnil}$ for any elements $a,b$ in a ring $R$. We extend this fact as follows.

\begin{lem} Let $R$ be a ring, and let $a,b,c\in R$. If $a(ba)^2=abaca=acaba=(ac)^2a$, then the following are equivalent:\end{lem}
\begin{enumerate}
\item [(1)]{\it $ac\in R^{qnil}$.}
\vspace{-.5mm}
\item [(2)]{\it $ba\in R^{qnil}$.}
\end{enumerate}\begin{proof} $\Longrightarrow$ By hypothesis, $a(ba)^2=(ac)^2a$ and $a(ba)^3=(ac)^3a$. Suppose that $ac\in R^{qnil}$. Let $y\in comm(ba)$. Then
$(1+yba)(1-yba+y^baba)=1-y^3bababa$, and so $$\begin{array}{ll}
&(1+yba)(1-yba+y^baba)(1+y^3bababa)\\
=&1-y^6babababababa\\
=&1-y^6b(acaca)bababa\\
=&1-y^6b(acac(ababa)ba\\
=&1-y^6b(acac(acaca)ba.
\end{array}$$
In view of Jacobson's Lemma (see \cite [Theorem 2.2]{C}), we will suffice to prove
$$1-abay^6bacacac(ac)\in U(R).$$ As $ac\in R^{qnil}$, we will suffice to check
$$abay^6bacacac(ac)=(ac)abay^6bacacac.$$
One easily checks that $$\begin{array}{lll}
abay^6bacacac(ac)&=&abay^6b(acacac)ac\\
&=&ay^6bababababac;\\
(ac)abay^6bacacac&=&(ac)ababay^6cacac\\
&=&(acacaca)y^6cacac\\
&=&(abababa)y^6cacac\\
&=&ay^6bababababac.
\end{array}$$
Hence $1+yba\in U(R)$. This shows that $ba\in R^{qnil}.$

$\Longleftarrow$ If $ba\in R^{qnil}$, by the preceding discussion, we see that $ab\in R^{qnil}$. With the same argument as above we get $ca\in R^{qnil}$, and therefore $ac\in R^{qnil}$.\end{proof}

We come now to the main result of this paper.

\begin{thm} Let $R$ be a ring, and let $a,b,c\in R$. If $a(ba)^2=abaca=acaba=(ac)^2a$, then the following are equivalent:\end{thm}
\begin{enumerate}
\item [(1)]{\it $ac\in R^{d}$.}
\vspace{-.5mm}
\item [(2)]{\it $ba\in R^{d}$.}
\end{enumerate}
In this case, $(ac)^{d}=a((ba)^{d})^2c$ and $(ba)^{d}=b((ac)^{d})^2a$.
\begin{proof}  Suppose that $ac$ has g-Drazin inverse and $(ac)^{d}=d$. Let $e=bd^2a$ and $f\in comm(ba)$. Then $$fe=fb((ac)^2d^3)^2a=fb(ac)^4d^6a=(ba)^4fcd^6a=b((ac)^3afc)d^6a.$$
Also we have
 $$\begin{array}{lll}
 ac((ac)^3afc)&=&(ac)^4afc=af(ba)^4c=af(ba)^3cac\\
&=&((ab)^3afc)ac=((ac)^3afc)ac.
 \end{array}$$
 Since $d\in comm^2(ac)$, we get $((ac)^3afc)d=d((ac)^3afc)$. Thus, we conclude that
 $$\begin{array}{lll}
 fe&=&b((ac)^3afc)d^6a=bd^6((ac)^3afc)a\\
 &=&bd^6(ab)^3afc=bd^6af(ba)^3ca\\
 &=&bd^6af(ba)^4=bd^6a(ba)^4f\\
 &=&bd^6a(ca)^4f=bd^2af=ef.
 \end{array}$$
 This implies that $e\in comm^2(ba)$.
 We have
 $$\begin{array}{lll}
 e(ba)e&=&bd^2a(ba)bd^2a=bd^2ababacd^3a\\
 &=&bd^2(ac)^3d^3a=bd^2a=e.
 \end{array}$$
 Let $p=1-acd$ then,
  $$pac=ac-acdac=ac-(ac)^2d$$
  that is contained in $R^{qnil}$. Moreover, we have
  $$\begin{array}{lll}
  ba-(ba)^2e&=&ba-bababd^2a=ba-bababacd^2da\\
  &=&ba-bacacacd^2da=b(1-acd)a=bpa.
  \end{array}$$
 One easily checks that
  $$\begin{array}{lll}
  abpabpa&=&ab(1-acd)ab(1-acd)a\\
  &=&ab(1-dac)aba(1-cda)\\
  &=&(ababa-abdacaba)(1-cda)\\
  &=&(abaca-abdacaca)(1-cda)\\
  &=&ab(1-dac)aca(1-cda)\\
  &=&ab(1-dac)ac(1-acd)a\\
  &=&abpacpa,
  \end{array}$$ and so
  $$(pa)b(pa)b(pa)=(pa)b(pa)c(pa).$$
  Likewise, we verify
  $$\begin{array}{c}
  (pa)b(pa)b(pa)=(pa)c(pa)b(pa)=(pa)c(pa)c(pa).
  \end{array}$$
  Then by Lemma 2.1., $bpa\in R^{qnil}$. Hence $ba$ has g-Drazin inverse $e$. That is,
  $e=bd^2a=(ba)^{d}.$ Moreover, we check $$\begin{array}{lll}
a((ba)^d)^2c&=&abd^2abd^2(ac)\\
&=&abd^3(acabac)d^2\\
&=&abd^3(acacac)d^2\\
&=&ab(acacac)d^5\\
&=&(ac)^4d^5\\
&=&(ac)^d,
\end{array}$$ as required.\end{proof}

\begin{cor} Let $R$ be a ring, let $k\in {\Bbb N}$, and let $a,b,c\in R$. If $a(ba)^2=abaca=acaba=(ac)^2a$, If $(ac)^k$ has g-Drazin inverse if and only if $(ba)^k$ has g-Drazin inverse.
\end{cor}
\begin{proof} Case 1. $k=1$. This is obvious by Theorem 2.2.

Case 2. $k=2$. We easily check that $$\begin{array}{lll}
a(bab)a(bab)a&=a(bab)a(cac)a\\
&=a(cac)a(bab)a\\
&=a(cac)a(cac)a.
\end{array}$$ The result follows by Theorem 2.2.

Case 3. $k\geq 3$. Then $(ac)^{k}=(ab)^{k-1}ac$. Hence, $(ac)^k$ has g-Drazin inverse if and only if $(ab)^{k}=(ac)(ab)^{k-1}$ has g-Drazin inverse.
This completes the proof.\end{proof}

\begin{cor} Let $R$ be a ring, and let $a,b,c\in R$. If $aba=aca$, then $ac\in R^{d}$ if and only if $ba\in R^{d}$. In this case, $(ba)^dc=b(ac)^d$.\end{cor}
\begin{proof} In view of Theorem 2.2., $ac\in R^{d}$ if and only if $ba\in R^{d}$. Moreover, $(ac)^{d}=a((ba)^{d})^2c$ and $(ba)^{d}=b((ac)^{d})^2a$.
Therefore $(ba)^dc=b((ac)^{d})^2ac=b(ac)^d$, as required.\end{proof}

\begin{lem} Let $R$ be a ring, and let $a\in R$. If $a\in R^{D}$, then $a\in R^{d}$ and $a^{D}=a^d$.\end{lem}
\begin{proof} This is obvious as the g-Drazin inverse of $a$ is unique.\end{proof}

\begin{lem} Let $R$ be a ring, and let $a,b,c\in R$. If $a(ba)^2=abaca=acaba=(ac)^2a$, then $ac\in R^{nil}$ if and only if $ba\in R^{nil}$.\end{lem}
\begin{proof} $\Longrightarrow$  Let $ac\in R^{nil}$, then there exists some $n\in {\Bbb N}$ such that $(ac)^n=0$. We may assume that $n$ is even. Hence $(ac)^na=(ac)^{n-2}(ac)^2a=(ac)^{n-2}a(ba)^2=(ac)^{n-4}(ac)^2a(ba)^2=(ac)^{n-4}a(ba)^{4}$ $=\cdots =(ac)^2a(ba)^{n-2}=a(ba)^n=0$ and so $(ba)^{n+1}=0$.

$\Longleftarrow$ It can be proved in the similar way. \end{proof}

\begin{thm} Let $R$ be a ring, and let $a,b,c\in R$. If $a(ba)^2=abaca=acaba=(ac)^2a$, then $ac\in R^{D}$ if and only if $ba\in R^{D}$. In this case, we have
$$(ba)^D=b(((ac)^D)^2)a, (ac)^D=a(((ba)^D)^2)c.$$\end{thm}
\begin{proof} Suppose that $ac\in R^{D}$. Then $ac\in R^{d}$ by Lemma 2.1. In view of Theorem 2.2, we see that $ba\in R^{d}$, and
$(ba)^d=b((ac)^d)^2a$. Let $p=1-(ac)(ac)^d$. As in the proof of Theorem 2.2, we have
$$\begin{array}{c}
(pa)b(pa)b(pa)=(pa)b(pa)c(pa)=(pa)c(pa)b(pa)=(pa)c(pa)c(pa);\\
(pa)c=ac-(ac)^2(ac)^D\in R^{nil}.\end{array}$$
In light of Lemma 2.6, $bpa\in R^{nil}$. Therefore $$\begin{array}{lll}
ba-(ba)^2(ba)^d&=&ba-babab((ac)^d)^2a\\
&=&ba-bababac((ac)^d)^3a\\
&=&ba-bacacac((ac)^d)^3a\\
&=&ba-b(ac)(ac)^da\\
&=&bpa\in R^{nil}.
\end{array}$$ Therefore $ba\in R^{D}$ and $(ac)^D=a(((ba)^D)^2)c.$ Moreover,
$(ac)^D=a(((ba)^D)^2)c.$ Conversely if $ba\in R^{D}$, then by \cite[Theorem 2.1]{LC}, $ab\in R^{D}$. Withe the same argument we get $ca\in R^D$ and so $ac\in R^{D}$.\end{proof}

Recall that $a$ has the group inverse if $a$ has Drazin inverse with index $1$, and denote  the group inverse by $a^{\#}$. As an immediate consequence of Theorem 2.7., we now derive

\begin{cor} Let $R$ be a ring, and let $a,b,c\in R$. If $a(ba)^2=abaca=acaba=(ac)^2a$, then $ac$ has group if and only if\end{cor}
\begin{enumerate}
\item [(1)]{\it $ba\in U(R)$; or}
\vspace{-.5mm}
\item [(2)]{\it $ba$ has group inverse and $(ba)^{\#}=b((ac)^{\#})^2a$; or}
\vspace{-.5mm}
\item [(3)]{\it $ba\in R^{D}$ and $(ba)^{D}=b((ac)^{\#})^2a$.}
\end{enumerate}

We note that if $aba=aca$ in a ring $R$ then $a(ba)^2=abaca=acaba=(ac)^2a$, but the converse is not true.

\begin{exam} Let $R=M_2({\Bbb Z}_2), x=\left(
\begin{array}{ccc}
0&1&0\\
0&0&1\\
0&0&0
\end{array}
\right)\in R$. Then $x^2\neq 0$ and $x^3=0$. Choose $$a=
\left(
\begin{array}{cc}
0&x\\
0&0
\end{array}
\right), b=\left(
\begin{array}{cc}
1&0\\
0&0
\end{array}
\right), c=\left(
\begin{array}{cc}
1&0\\
1&1
\end{array}
\right).$$ Then $a(ba)^2=abaca=acaba=(ac)^2a$, but $aba\neq aca$. In this case, $ac\in R^{D}$.\end{exam}

\section{Common spectral properties of bounded linear operators}

Let $A$ be a complex Banach algebra with unity $1$, and let $a\in A$. The Drazin spectrum $\sigma_D(a)$ and g-Drazin spectrum $\sigma_{d}(a)$ are defined by $$\begin{array}{c}
\sigma_D(a)=\{ \lambda\in {\Bbb C}~|~\lambda-a\not\in A^{D}\};\\
\sigma_{d}(a)=\{ \lambda\in {\Bbb C}~|~\lambda-a\not\in A^{d}\}.
\end{array}$$ Let $X$ be  Banach space, and let $L(X)$ denote the set of all bounded linear operators from Banach space to itself. The goal of this section is concern on common spectrum properties of $L(X)$. The following lemma is crucial.

\begin{lem} Let $R$ be a ring, and let $a,b,c\in R$. If $a(ba)^2=abaca=acaba=(ac)^2a$, then $$1-ac\in U(R)\Longleftrightarrow 1-ba\in U(R).$$
\end{lem}
\begin{proof} $\Longrightarrow$ Write $s(1-ac)=(1-ac)s=1$ for some $s\in R$. Then $sac=s-1$. We see that
$$\begin{array}{ll}
&\big((1+bsa)(1+ba)-bsa\big)(1-ba)\\
=&(1+bsa)(1-baba)-bsa(1-ba)\\
=&1-baba+bsa-bsababa-bsa(1-ba)\\
=&1-baba+bsa-bsacaba-bsa(1-ba)\\
=&1-baba+bsa-b(s-1)aba-bsa(1-ba)\\
=&1.
\end{array}$$ Thus, $1-ba\in R$ is left invertible. Likewise, we see that it is right invertible. Therefore $$(1-ba)^{-1}=\big(1+b(1-ac)^{-1}a\big)(1+ba)-b(1-ac)^{-1}a,$$ as asserted.

$\Longleftarrow$ This is symmetric.\end{proof}

\begin{thm} Let $A,B,C\in L(X)$ such that $A(BA)^2=ABACA=ACABA=(AC)^2A$, then $$\sigma_d(AC)=\sigma_d(BA).$$\end{thm}
\begin{proof} Case 1. $0\in \sigma_d(AC)$. Then $AC\not\in A^{d}$. In view of Theorem 2.2., $BA\not\in A^{d}$. Thus $0\in \sigma_d(BA)$.

Case 2. $0\not\in \lambda\in\sigma_d(AC)$. Then $\lambda\in acc\sigma(AC)$. Thus, we see that
$$\lambda=\lim\limits_{n\to \infty}\{ \lambda_n ~|~ \lambda_n I-AC\not\in L(X)^{-1}\}.$$
For $\lambda_n\neq 0$, it follows by Lemma 3.1 that $I-(\frac{1}{\lambda_n} A)C\in L(X)^{-1}$ if and only if $I-B(\frac{1}{\lambda_n} A)\in L(X)^{-1}$. Therefore
$$\lambda=\lim\limits_{n\to \infty}\{ \lambda_n ~|~ \lambda_n I-BA\not\in L(X)^{-1}\}\in acc\sigma(BA)=\sigma_d(BA).$$
Therefore $\sigma_d(AC)\subseteq \sigma_d(BA).$ Likewise, $\sigma_d(BA)\subseteq \sigma_d(AC)$, as required.\end{proof}

\begin{cor} Let $A,B,C\in L(X)$ such that $ABA=ACA$, then $$\sigma_d(AC)=\sigma_d(BA).$$\end{cor}
\begin{proof} This is obvious by Theorem 3.2.\end{proof}

\begin{exam} Let $A,B,C$ be operators, acting on separable Hilbert space $l_2({\Bbb N})$, defined as follows respectively:
$$\begin{array}{lll}
A(x_1,x_2,x_3,x_4,\cdots )&=&(0,x_2,0,x_4,\cdots ),\\
B(x_1,x_2,x_3,x_4,\cdots )&=&(0,x_1,x_2,x_4,\cdots ),\\
C(x_1,x_2,x_3,x_4,\cdots )&=&(0,0,x_1,x_4,\cdots ).
\end{array}$$ Then $ABA=ACA$, and so $\sigma_d(AC)=\sigma_d(BA)$ by Corollary 3.3.\end{exam}

For the Drazin spectrum $\sigma_D(a)$, we now derive

\begin{thm} Let $A,B,C\in L(X)$ such that $A(BA)^2=ABACA=ACABA=(AC)^2A$, then $$\sigma_D(AC)=\sigma_D(BA).$$\end{thm}
\begin{proof} In view of Theorem 2.7, $AC\in L(X)^{D}$ if and only if $BA\in L(X)^{D}$, and therefore we complete the proof by~\cite[Theorem 3.1]{Y}.\end{proof}

A bounded linear operator $T\in L(X)$ is Fredholm operator if $dimN(T)$ and $codimR(T)$ are finite, where $N(T)$ and $R(T)$ are the null space and the range of $T$ respectively. If furthermore the Fredholm index $ind(T)=0$, then $T$ is said to be Weyl operator. For each nonnegative integer $n$ define $T_{|n|}$ to be the restriction of $T$ to $R(T^n)$. If for some $n$, $R(T^n)$ is closed and $T_{|n|}$ is a Fredholm operator then $T$ is called a $B$-Fredholm operator. $T$ is said to be a $B$-Weyl operator if $T_{|n|}$ is a Fredholm operator of index zero (see \cite{Ba}). The $B$-Fredholm and $B$-Weyl spectrums of $T$ are defined by
$$\begin{array}{c}
\sigma_{BF}(T)=\{ \lambda\in {\Bbb C}~|~T-\lambda I~\mbox{is not a}$ $B-\mbox{Fredholm operator}\};\\
\sigma_{BW}(T)=\{ \lambda\in {\Bbb C}~|~T-\lambda I~\mbox{is not a}$ $B-\mbox{Weyl operator}\}.
\end{array}$$

\begin{cor} Let $A,B,C\in L(X)$ such that $A(BA)^2=ABACA$ $=ACABA=(AC)^2A$, then $$\sigma_{BF}(AC)=\sigma_{BF}(BA).$$\end{cor}
\begin{proof} Let $\pi: L(X)\to L(X)/F(X)$ be the canonical map and $F(X)$ be the ideal of finite rank operators in $L(X)$. As in well known, $T\in L(X)$ is $B$-Fredholm if and only if $\pi(T)$ had Drazin inverse. By hypothesis, we see that $$\begin{array}{lll}
\pi(A)(\pi(B)\pi(A))^2&=&\pi(A)\pi(B)\pi(A)\pi(C)\pi(A)\\
&=&\pi(A)\pi(C)\pi(A)\pi(B)\pi(A)\\
&=&(\pi(A)\pi(C))^2\pi(A).
\end{array}$$ According to Theorem 3.5., for every scalar $\lambda$, we have
$$\lambda I-\pi(AC)~\mbox{has Drazin inverse} ~\Longrightarrow \lambda I-\pi(BD)~\mbox{has Drazin inverse}.$$ This completes the proof.\end{proof}

\begin{cor} Let $A,B,C\in L(X)$ such that $A(BA)^2=ABACA=ACABA=(AC)^2A$, then $$\sigma_{BW}(AC)=\sigma_{BW}(BA).$$\end{cor}
\begin{proof} If $T$ is $B$-Fredholm then for $\lambda\neq 0$ small enough, $T-\lambda I$ is Fredholm and $ind(T)=ind(T-\lambda I)$.
As in the proof of \cite[Lemma 2.3, Lemma 2.4]{Y}, we see that $I-AC$ is Fredholm if and only if $I-BD$ is Fredholm and in this case, $ind(I-AC)=ind(I-BD)$. Therefore we complete the proof by
by Corollary 3.6.\end{proof}

\end{document}